\documentclass[11pt]{amsart}

\usepackage{amsmath,amssymb,amsthm}
\usepackage{mathtools}
\usepackage[hidelinks]{hyperref}

\theoremstyle{plain}
\newtheorem{theorem}{Theorem}[section]
\newtheorem{proposition}[theorem]{Proposition}
\newtheorem{lemma}[theorem]{Lemma}
\newtheorem{corollary}[theorem]{Corollary}

\theoremstyle{definition}
\newtheorem{remark}[theorem]{Remark}

\DeclareMathOperator{\Nm}{N}
\newcommand{\Z}{\mathbb{Z}}
\newcommand{\C}{\mathbb{C}}

\title[The $5$-divisible determinants for $C_5^2$]
{The $5$-divisible integer group determinants for the elementary abelian group of order~25}

\author{Chatchawan Panraksa}

\address{Applied Mathematics,
Mahidol University International College,
Nakhon Pathom 73170, Thailand}

\email{chatchawan.pan@mahidol.ac.th}

\date{}

\begin{document}

\begin{abstract}
Let $G=C_5\times C_5$, and let $S(G)$ denote the set of integer values of
its group determinant. Previous work determines the values in $S(G)$
coprime to~$5$ and proves that every $5$-divisible value is divisible by
$5^8$. We prove the converse inclusion
\[
  5^8\Z\subseteq S(G).
\]
Consequently,
\[
  S(G)=\{m\in\Z:m\equiv\pm1\text{ or }\pm7\pmod{25}\}\cup5^8\Z.
\]
The proof uses a general shift criterion and three explicit polynomials
whose group determinants are $5^8$, $2\cdot5^8$, and $5^9$.
Together with the known classification for $C_{25}$, this completes the
Taussky--Todd integer group determinant problem for all groups of order~$25$.
\end{abstract}



\maketitle

\section{Introduction}\label{sec:introduction}

For a finite group $G$, the group determinant introduced by
Frobenius~\cite{Frobenius1896} is
\[
  \Theta_G((x_g)_{g\in G})=\det(x_{gh^{-1}})_{g,h\in G}.
\]
When $G$ is abelian, Dedekind's factorization gives
\begin{equation}\label{eq:dedekind}
  \Theta_G((x_g))
  =\prod_{\chi\in\widehat G}\sum_{g\in G}x_g\chi(g),
  \qquad
  \widehat G:=\operatorname{Hom}(G,\C^\times);
\end{equation}
see also~\cite[pp.~420--421]{Dedekind1896}. We write
\[
  S(G):=\{\Theta_G((a_g)):a_g\in\Z\}\subseteq\Z.
\]
The problem of determining $S(G)$ is often called the
Taussky--Todd integer group determinant problem. This problem is now
solved for every group of order less than~20
\cite{PinnerSmyth2020,PaudelPinner2022Order15,YamaguchiYamaguchi2024Order16,PaudelPinner2025Order18}.

For an odd prime $p$, put $G_p=C_p\times C_p$ and
$\zeta_p=e^{2\pi i/p}$. If
\[
  F(x,y)=\sum_{r,s=0}^{p-1}a_{r,s}x^r y^s\in\Z[x,y],
\]
then the corresponding group determinant is
\begin{equation}\label{eq:measure}
  M_p(F):=\prod_{i,j=0}^{p-1}F(\zeta_p^i,\zeta_p^j).
\end{equation}
Thus $S(G_p)=\{M_p(F):F\in\Z[x,y]\}$, after reducing exponents modulo~$p$.

Two earlier results settle the necessary conditions for $G_5$.
First, DeSilva and Pinner~\cite[Lemmas~2.1 and~2.2]{DeSilvaPinner2014}
proved that the values for $C_p^n$ coprime to~$p$ are precisely
\[
  \{m\in\Z:m^{p-1}\equiv1\pmod{p^n}\}.
\]
For $p=5$ and $n=2$, these are the four residue classes
$m\equiv\pm1,\pm7\pmod{25}$. Second, Mossinghoff and
Pinner~\cite{MossinghoffPinner2024} proved that
\begin{equation}\label{eq:known-divisibility}
  p\mid M_p(F)\quad\Longrightarrow\quad p^{p+3}\mid M_p(F).
\end{equation}
They also showed that every possible $p$-adic valuation at least $p+3$
occurs. For $p=5$,~\eqref{eq:known-divisibility} gives the necessary
condition $5^8\mid M_5(F)$ whenever $5\mid M_5(F)$, but it does not
determine which cofactors of $5^8$ occur.

The new result of this paper is that every cofactor occurs.

\begin{theorem}\label{thm:divisible-image}
For $G_5=C_5\times C_5$,
\[
  S(G_5)\cap5\Z=5^8\Z.
\]
\end{theorem}

Combining Theorem~\ref{thm:divisible-image} with the result of
DeSilva and Pinner gives the complete integer image.

\begin{corollary}\label{cor:full-image}
For $G_5=C_5\times C_5$,
\[
  S(G_5)
  =\{m\in\Z:m\equiv\pm1\text{ or }\pm7\pmod{25}\}\cup5^8\Z.
\]
\end{corollary}

Since every group of order~$25$ is isomorphic to either $C_{25}$ or
$C_5\times C_5$, and Mossinghoff and Pinner determined $S(C_{25})$
in~\cite{MossinghoffPinner2023PrimePower},
Corollary~\ref{cor:full-image} completes the Taussky--Todd integer group
determinant problem for all groups of order~$25$.

The proof is organized around a shift identity valid for every odd
prime. It reduces the inclusion $5^8\Z\subseteq S(G_5)$ to three
explicit ``seed'' determinants. Section~\ref{sec:shift} develops this
criterion, Section~\ref{sec:seeds} verifies the three seeds, and
Section~\ref{sec:proof} completes the proof and records the limitation
of the method for larger primes.

\section{A shift identity and a seed criterion}\label{sec:shift}

Let
\[
  \Phi_p(t)=1+t+\cdots+t^{p-1},
  \qquad H_p(x,y)=\Phi_p(x)\Phi_p(y).
\]
For $0\le i,j<p$,
\begin{equation}\label{eq:H-values}
  H_p(\zeta_p^i,\zeta_p^j)
  =\begin{cases}
      p^2,&(i,j)=(0,0),\\
      0,&(i,j)\ne(0,0).
    \end{cases}
\end{equation}

\begin{lemma}[Shift identity]\label{lem:shift}
Let $F\in\Z[x,y]$ and $c\in\Z$. Then
\begin{equation}\label{eq:shift}
  M_p(F+cH_p)
  =\bigl(F(1,1)+p^2c\bigr)
    \prod_{(i,j)\ne(0,0)}F(\zeta_p^i,\zeta_p^j).
\end{equation}
If $F(1,1)\ne0$, this becomes
\begin{equation}\label{eq:shift-ratio}
  M_p(F+cH_p)
  =\frac{F(1,1)+p^2c}{F(1,1)}M_p(F).
\end{equation}
Moreover, when $p$ is odd,
\begin{equation}\label{eq:negation}
  M_p(-F)=-M_p(F).
\end{equation}
\end{lemma}

\begin{proof}
Equation~\eqref{eq:H-values} shows that adding $cH_p$ changes only the
factor indexed by $(i,j)=(0,0)$ in~\eqref{eq:measure}, proving
\eqref{eq:shift} and~\eqref{eq:shift-ratio}. Since $M_p$ is homogeneous
of degree $p^2$, and $p^2$ is odd, $M_p(-F)=(-1)^{p^2}M_p(F)=-M_p(F)$.
\end{proof}

The following criterion isolates the constructive part of the problem.

\begin{proposition}[Seed criterion]\label{prop:seed-criterion}
Let $p$ be odd. Suppose that $A\subset\Z\setminus\{0\}$ maps
surjectively onto $\Z/p\Z$ and that, for each $a\in A$, there is a
polynomial $F_a\in\Z[x,y]$ satisfying
\begin{equation}\label{eq:seed-condition}
  F_a(1,1)=pa,
  \qquad
  M_p(F_a)=a\,p^{p+3}.
\end{equation}
Then
\[
  p^{p+3}\Z\subseteq S(G_p).
\]
\end{proposition}

\begin{proof}
For $a\in A$, Lemma~\ref{lem:shift} gives
\[
  M_p(F_a+cH_p)
  =\frac{pa+p^2c}{pa}\,a p^{p+3}
  =p^{p+3}(a+pc).
\]
As $a$ runs through representatives of all residue classes modulo~$p$
and $c$ runs through $\Z$, the integers $a+pc$ exhaust~$\Z$.
\end{proof}

For $p=5$, it is enough to construct seeds for $a=1,2,5$:
by~\eqref{eq:negation}, the first two also give seeds for $a=-1,-2$, and
\[
  \{1,-1,2,-2,5\}\longrightarrow\Z/5\Z
\]
is surjective.

\section{Three seed determinants for \texorpdfstring{$C_5^2$}{C5 squared}}\label{sec:seeds}

Write $\zeta=\zeta_5$, $K=\mathbb{Q}(\zeta)$, and $\pi=1-\zeta$. The $24$
nontrivial pairs in $(\Z/5\Z)^2$ split into six orbits under
multiplication by $(\Z/5\Z)^\times$. We use the representatives
\begin{equation}\label{eq:reps}
  \mathcal R=\{(1,0),(0,1),(1,1),(3,1),(3,4),(2,3)\}.
\end{equation}
Consequently,
\begin{equation}\label{eq:norm-factorization}
  M_5(F)=F(1,1)
  \prod_{(a,b)\in\mathcal R}
  \Nm_{K/\mathbb{Q}}\bigl(F(\zeta^a,\zeta^b)\bigr).
\end{equation}
We shall use
\begin{equation}\label{eq:basic-norms}
  \Nm_{K/\mathbb{Q}}(\zeta)=1,
  \qquad
  \Nm_{K/\mathbb{Q}}(1+\zeta)=\Nm_{K/\mathbb{Q}}(1+\zeta^2)=1,
  \qquad
  \Nm_{K/\mathbb{Q}}(\pi)=5.
\end{equation}
Indeed, the middle identities follow from $\Phi_5(-1)=1$ and Galois
conjugacy.

Define
\begin{align}
  F_1(x,y)&=1+y+y^2+x+x^4y,\label{eq:F1}\\
  F_2(x,y)&=1+y+y^2+y^3+x+xy+xy^2
             +x^4+x^4y+x^4y^4,\label{eq:F2}\\
  F_5(x,y)&=H_5(x,y)-1-y+x+x^4.\label{eq:F5}
\end{align}
Their coefficient sums are
\begin{equation}\label{eq:sums}
  F_1(1,1)=5,
  \qquad F_2(1,1)=10,
  \qquad F_5(1,1)=25.
\end{equation}

\begin{proposition}\label{prop:three-seeds}
The polynomials in~\eqref{eq:F1}--\eqref{eq:F5} satisfy
\[
  M_5(F_1)=5^8,
  \qquad
  M_5(F_2)=2\cdot5^8,
  \qquad
  M_5(F_5)=5^9.
\]
Thus they satisfy the seed condition~\eqref{eq:seed-condition} for
$a=1,2,5$, respectively.
\end{proposition}

\begin{proof}
Direct reduction using $1+\zeta+\zeta^2+\zeta^3+\zeta^4=0$ gives the
following exact values. Each row is indexed by one representative from
\eqref{eq:reps}.
\[
\renewcommand{\arraystretch}{1.22}
\begin{array}{c|c|c|c}
(a,b)&F_1(\zeta^a,\zeta^b)&F_2(\zeta^a,\zeta^b)&F_5(\zeta^a,\zeta^b)\\ \hline
(1,0)&-\zeta^3\pi^2(1+\zeta)^2
     &\pi^2(1+\zeta)^3
     &\zeta^4\pi^2\\
(0,1)&-\zeta^3\pi(1+\zeta)^2
     &-\zeta^3\pi(1+\zeta)^2
     &\pi\\
(1,1)&-\zeta^3\pi(1+\zeta)^2
     &\zeta^3\pi
     &\zeta^4\pi\\
(3,1)&\zeta^3\pi
     &\zeta^3\pi
     &-\pi(1+\zeta)^2\\
(3,4)&-\zeta^2\pi
     &\zeta^2\pi(1+\zeta)^2
     &\zeta^2\pi(1+\zeta)^2\\
(2,3)&-\zeta^4\pi(1+\zeta)
     &-\zeta^2\pi(1+\zeta^2)
     &-\pi(1+\zeta)
\end{array}
\]
For example,
\[
  F_1(\zeta,1)=3+\zeta+\zeta^4
  =-\zeta^3(1-\zeta)^2(1+\zeta)^2;
\]
the other entries follow in the same way.

By~\eqref{eq:basic-norms}, the norm in the first row of each column is
$25$, while the norm in each of the remaining five rows is~$5$.
Therefore the product of the six nontrivial orbit norms is
$25\cdot5^5=5^7$ for each polynomial. Combining this with
\eqref{eq:norm-factorization} and~\eqref{eq:sums} gives
\[
  M_5(F_1)=5\cdot5^7=5^8,
  \quad
  M_5(F_2)=10\cdot5^7=2\cdot5^8,
  \quad
  M_5(F_5)=25\cdot5^7=5^9.
\]
\end{proof}

\section{Proof of the main result}\label{sec:proof}

\begin{proof}[Proof of Theorem~\ref{thm:divisible-image}]
The containment
\[
  S(G_5)\cap5\Z\subseteq5^8\Z
\]
is the case $p=5$ of~\eqref{eq:known-divisibility}, proved in
\cite{MossinghoffPinner2024}.

For the reverse containment, Proposition~\ref{prop:three-seeds} supplies
seeds for $a=1,2,5$. Lemma~\ref{lem:shift}, applied to $-F_1$ and
$-F_2$, supplies seeds for $a=-1,-2$. Proposition~\ref{prop:seed-criterion}
therefore gives
\[
  5^8\Z\subseteq S(G_5).
\]
The two containments prove the theorem.
\end{proof}

\begin{proof}[Proof of Corollary~\ref{cor:full-image}]
DeSilva and Pinner~\cite[Lemmas~2.1 and~2.2]{DeSilvaPinner2014} show
that the elements of $S(G_5)$ coprime to~$5$ are exactly the integers
satisfying $m^4\equiv1\pmod{25}$, namely
$m\equiv\pm1,\pm7\pmod{25}$. The $5$-divisible elements are exactly
$5^8\Z$ by Theorem~\ref{thm:divisible-image}.
\end{proof}

\begin{remark}[Scope of the construction]\label{rem:scope}
For $G_3=C_3\times C_3$, Pinner and Smyth proved the analogous formula
\[
  S(G_3)=\{9m\pm1:m\in\Z\}\cup3^6\Z
\]
\cite[Theorem~6.1]{PinnerSmyth2020}. The shift identity and seed
criterion above are valid for every odd prime, but the existence of a
set of seeds satisfying~\eqref{eq:seed-condition} is a separate
problem. The present argument is specific to $p=5$, and the clean
formulas for $p=3$ and $p=5$ may reflect exceptional behavior. In
particular, we do not assert that
$p^{p+3}\Z\subseteq S(C_p\times C_p)$ for any $p\ge7$.
\end{remark}

\section*{Acknowledgments}

The author thanks the referee for pointing out the work of DeSilva and
Pinner and for suggestions that led to a shorter and more focused
presentation.

\end{document}